\documentclass{article}

\usepackage{lineno,hyperref}
\usepackage{amssymb,amsmath, amsfonts, amsthm}

\usepackage{graphicx}
\usepackage{pgfplots}
\usepackage{tikz}
\usepackage{booktabs}
\usepackage{array}
\usepackage{verbatim}
\usepackage{subfig}
\usepackage{geometry}
\usepackage{fancyhdr}

\newcommand{\smallqed}{{\tiny ($\Box$)}}

\newcommand{\vertex}{\node[vertex]}
\tikzstyle{vertex}=[circle, draw, inner sep=0pt, minimum size=6pt]
\newtheorem{theorem}{Theorem}

\newtheorem{prop}{Proposition}
\newtheorem{conjecture}{Conjecture}

\newtheorem{observation}{Observation}

\modulolinenumbers[5]

\begin{document}

\title{Claw-free cubic graphs are $(1, 1, 2, 2)$-colorable}
\author{
Bo\v{s}tjan Bre\v{s}ar$^{a,b}$
\and
Kirsti Kuenzel$^{c}$
\and
Douglas F. Rall$^{d}$\\
}

\date{}

\maketitle

\begin{center}
$^a$ Faculty of Natural Sciences and Mathematics, University of Maribor, Slovenia\\

$^b$ Institute of Mathematics, Physics and Mechanics, Ljubljana, Slovenia\\
$^c$ Department of Mathematics, Trinity College, Hartford, CT, USA\\
$^d$ Department of Mathematics, Furman University, Greenville, SC, USA\\
\end{center}
\vskip15mm

\begin{abstract}
A $(1,1,2,2)$-coloring of a graph is a partition of its vertex set into four sets two of which are independent and the other two are $2$-packings. In this paper, we prove that every claw-free cubic graph admits a $(1,1,2,2)$-coloring. This implies that the conjecture from [Packing chromatic number, $(1,1,2,2)$-colorings, and characterizing the Petersen graph,  Aequationes Math.\ 91 (2017) 169--184] that the packing chromatic number of subdivisions of subcubic graphs is at most $5$ is true in the case of claw-free cubic graphs.

\end{abstract}

\section{Introduction} \label{sec:intro}

Given a non-decreasing sequence $S=(a_1,a_2,\ldots,a_r)$ of positive integers, a function $f:V(G)\rightarrow [r]$ is an {\em $S$-packing coloring} of $G$ if for every two vertices $u,v\in V(G)$ with $f(u)=i=f(v)$, the distance between $u$ and $v$ is greater than $a_i$. The concept was first mentioned in~\cite{goddard-2008}, where the main focus was on the so-called packing coloring, which is considered when $S=(1,2,\ldots,r)$, while in~\cite{goddard-2012} more extensive  investigations of $S$-packing colorings were initiated. The {\em packing chromatic number} of a graph is the smallest positive integer $r$ for which the graph admits a $(1,2,\ldots,r)$-packing coloring.
The concepts were studied in a number of papers; see the survey~\cite{BFKR} and the references therein. Note that if $S$ consists only of $1$s, we get the standard (vertex) coloring, while if $S$ is the constant sequence with all integers $d$, we get the $d$-distance coloring of graphs (the latter colorings were surveyed in~\cite{kramer}).  We remark that packing colorings combine colorings and packings, where a {\em $k$-packing} is a set of vertices every pair of which is at distance greater than $k$.

In the studies of $S$-packing colorings, a lot of attention has been given to the class of graphs with maximum degree $3$ (i.e., {\em subcubic graphs}),
and their subclass of cubic graphs. Gastineau and Togni found several sequences $S$ with integers in $\{1,2\}$ for which subcubic graphs are $S$-packing colorable~\cite{gt-2016}. In addition, they posed a question whether every subcubic graph except for the Petersen graph is $(1,1,2,3)$-packing colorable. This question and its variants were  considered in several papers.  Kostochka and Liu proved that every outerplanar subcubic graph is $(1,1,2,4)$-packing colorable~\cite{Kostochka-2021}.
Bre\v sar, Gastineau and Togni proved that every triangle-free, outerplanar, subcubic graph is $(1,2,2,2)$-packing colorable~\cite{bgt}, while $(1,2,2,2)$-packing colorability was also established for all $2$-connected, outerplanar, subcubic graphs by Yang and Wu~\cite{YW}. Recently, Tarhini and Togni proved that cubic Halin graphs are $(1,1,2,3)$-packing colorable~\cite{tt-2024}.

Bre\v sar, Klav\v zar, Rall, and Wash~\cite{bkrw-2017b} studied packing colorings of the subdivisions of subcubic graphs.  Let $S(G)$ stand for the subdivision of a graph $G$, which is the graph obtained from $G$ by subdividing every edge by one vertex. Based on a question of Gastineau and Togni~\cite{gt-2016}, the following conjecture was posed in~\cite{bkrw-2017b}:
\begin{conjecture}
\label{conj1}
If $G$ is a subcubic graph, then $S(G)$ admits a $(1,2,3,4,5)$-packing coloring.
\end{conjecture}
Equivalently, the conjecture states that every subdivision of a subcubic graph has the packing chromatic number at most $5$. Balogh, Kostochka, and Liu proved in~\cite{bkl} that the packing chromatic number of the subdivisions of subcubic graphs is at most $8$. Another approach to the conjecture has been to confirm its correctness in some subclasses of subcubic graphs.
By a result of Gastineau and Togni from~\cite{gt-2016}, if a graph admits a $(1,1,2,2)$-packing coloring, then its subdivision admits a $(1,2,3,4,5)$-packing coloring, which gives an additional motivation for studying $(1,1,2,2)$-packing colorings in subcubic graphs. In~\cite{bkrw-2017b} it was proved that generalized prisms of a cycle (with the exception of the Petersen graph) admit a $(1,1,2,2)$-packing coloring, hence Conjecture~\ref{conj1} holds for this kind of cubic graph.  Liu et al.~\cite{liu-2020} established that every subcubic graph with maximum average degree at most $30/11$ is $(1,1,2,2)$-packing colorable.
Recently, Mortada and Togni~\cite{mt-2024} proved that a subcubic graph $G$ is $(1,1,2,2)$-packing colorable if there are no two adjacent degree-$3$ vertices all of whose neighbors also have have degree $3$.

In this paper, we utilize the knowledge about the structure of cubic claw-free graphs that can be found in the literature. (Recall that the {\em claw} is the star $K_{1,3}$, and {\em claw-free graphs} are those with no claw as an induced subgraph.)  We establish that every cubic claw-free graph is $(1,1,2,2)$-packing colorable. This, in particular, implies the truth of Conjecture~\ref{conj1} in the case of cubic claw-free graphs.


\section{Main Result}
A {\it diamond} is a graph $D$ obtained from the complete graph $K_4$ by deleting an edge.
Given a diamond $D$, we refer to the two vertices of degree $3$ in $D$ as the \emph{internal vertices} of $D$,  and the remaining two vertices as the \emph{external vertices} of $D$.

For a positive integer $k$, a {\it string of $k$ diamonds} is the graph obtained from the disjoint union of $k$ diamonds $D_1, \dots, D_k$ and two additional isolated vertices $x_0$ and $y_{k+1}$, where $V(D_i)=\{x_i,y_i,z_i,w_i\}$ for all $i\in [k]$, and $z_i$ and $w_i$ are internal vertices in $D_i$, by adding the edges $x_{i-1}y_i$ for all $i\in [k+1]$.
Furthermore, a {\it ring of diamonds} is defined to be any connected, claw-free cubic graph in which every vertex is in a diamond. Note that a ring of diamonds can be obtained from a string of diamonds by removing the vertices $x_0$ and $y_{k+1}$ and joining $y_1$ and $x_k$ by an edge.

Let $G$ be a graph. By {\em replacing an edge with a string of $k$ diamonds} we mean the operation by which an edge $xy\in E(G)$ is removed from the graph, and the vertex $x_0$ from a string of $k$ diamonds $D$ is identified with $x$, while $y_{k+1}$ from $D$ is identified with $y$.

The following result of Oum~\cite{Oum} will be instrumental in our proofs.

\begin{theorem}{\rm \cite{Oum}}\label{thm:2factor} A graph $G$ is $2$-edge-connected, claw-free cubic if and only if either
\begin{enumerate}
\item[(i)] $G \cong K_4$,
\item[(ii)] $G$ is a ring of diamonds, or
\item[(iii)] $G$ can be built from a $2$-edge-connected, cubic multigraph $H$ by replacing some edges of $H$ with strings of diamonds and replacing each vertex of $H$ with a triangle.
\end{enumerate}
\end{theorem}

When considering a $2$-edge-connected, claw-free cubic graph $G$, which satisfies (iii) in Theorem~\ref{thm:2factor}, we will often speak about the {\em underlying cubic multigraph} $H$ from which $G$ can be built by replacing each vertex of $H$ with a triangle and (possibly) some edges of $H$ with strings of diamonds. We also recall that Petersen \cite{Petersen} showed that $H$ contains a $2$-factor:

\begin{theorem}{\rm \cite{Petersen}}\label{thm:Petersen} Every bridgeless, cubic multigraph contains a $2$-factor.
\end{theorem}

Note that if $G$ is a $2$-edge-connected, claw-free graph built from a corresponding cubic multigraph $H$, we can build a $2$-factor for $G$ by taking any $2$-factor from the corresponding cubic multigraph $H$ by replacing each vertex with a triangle and some edges with strings of diamonds. In particular, if an edge of $H$ does not belong to a $2$-factor of $H$, yet it is replaced in $G$ by a string of diamonds, then in each of the diamonds the corresponding spanning $4$-cycles can be added to the $2$-factor of $G$.

In the rest of the paper, we will omit the word ``packing'' in ``packing coloring'', and write shortly ``$(1,1,2,2)$-coloring.'' In addition, when proving that a graph has a $(1,1,2,2)$-coloring we will assign colors from the set $\{1_a,1_b,2_a,2_b\}$, instead of defining a coloring function.  Any pair of vertices that are both assigned the color $1_a$ or both assigned $1_b$ must be nonadjacent.  The distance in the graph between any two vertices that are both assigned $2_a$ or both assigned $2_b$ must be at least $3$.

\begin{prop}\label{prop:diamonds} If $G$ is a ring of diamonds, then $G$ is $(1, 1, 2, 2)$-colorable.
\end{prop}

\begin{proof}
One can obtain the corresponding coloring by dividing the vertex set of $G$ into quadruples of vertices each of which induce a diamond. Now, the two vertices of a quadruple $Q_i$ all of whose neighbors are in $Q_i$ receive colors $2_a$ and $2_b$, while if a vertex $x$ of $Q_i$ is adjacent to a vertex $y$ of $Q_j$, where $i\ne j$, then $x$ and $y$ get distinct colors from $\{1_a,1_b\}$.
\end{proof}

We also utilize the following result due to Plesn\'{i}k.

\begin{theorem}{\rm \cite{Plesnik}}\label{thm:Plesnik} Let $G$ be an $(r-1)$-edge-connected, regular graph of degree $r>0$ where $|V(G)|$ is even and let $H$ be an arbitrary set of $r-1$ edges. The graph $G' = G - H$ has a $1$-factor.
\end{theorem}

As mentioned in~\cite{Plesnik}, Theorem~\ref{thm:Plesnik} holds also in the case of multigraphs. Consider the specific case when $G$ is a $2$-edge-connected, cubic multigraph where $|V (G)|$ is even.  Let $e$ be any edge of $G$ and $f \in E(G)- \{e\}$. By Theorem~\ref{thm:Plesnik}, $G- \{e, f\}$ has a $1$-factor, say $F$. Note that $G -F$ is $2$-regular (and so is a disjoint union of cycles), and $\{e, f\} \subseteq E(G -F)$. Therefore, $e$ belongs to one of these cycles. We state this result, a corollary of Theorem~\ref{thm:Plesnik}, which we will use repeatedly in the remainder of the paper.

\begin{theorem}\label{thm:multigaph-1factor} Let $G$ be a $2$-edge-connected, cubic multigraph where $|V(G)|$ is even. For an arbitrary edge $e \in E(G)$, $G$ has a $2$-factor that contains $e$.
\end{theorem}

We begin by first considering  claw-free cubic graphs that are $2$-edge-connected. We use similar notation in the following proof to that used in \cite{AK-2022}.

\begin{theorem}\label{thm:2connected} If $G$ is a $2$-edge-connected, claw-free cubic graph, then $G$ is $(1, 1, 2, 2)$-colorable.
\end{theorem}

\begin{proof} Since $K_4$ is clearly $(1, 1, 2, 2)$-colorable, we may assume by Theorem~\ref{thm:2factor} and Proposition~\ref{prop:diamonds} that $G$ can be obtained from a $2$-edge-connected, cubic multigraph $H$ by replacing each vertex of $H$ with a triangle and possibly replacing some edges of $H$ with strings of diamonds. By Theorem~\ref{thm:Petersen}, $H$ contains a $2$-factor $\mathcal{C'}$ that is a disjoint union of cycles, say $ \mathcal{C'}= C_1' \cup \cdots \cup C_k'$.

  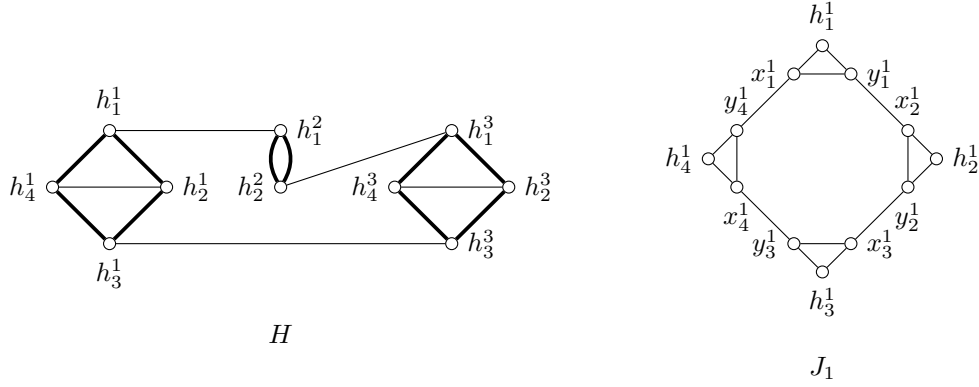
\begin{figure}[h]
\begin{center}
\begin{tikzpicture}[scale=.75]
\tikzstyle{vertex}=[circle, draw, inner sep=0pt, minimum size=6pt]
\tikzset{vertexStyle/.append style={rectangle}}
	\vertex (1) at (-1,0) [scale=.75, label=below:$h_3^1$] {};
	\vertex (2) at (-2, 1) [ scale=.75, label=left:$h_4^1$] {};
	\vertex (3) at (-1,2) [scale=.75, label=above:$h_1^1$] {};
	\vertex (4) at (0,1) [scale=.75, label=right:$h^1_2$] {};
	\vertex (5) at (2, 2) [scale=.75, label=right:$h^2_1$] {};
	\vertex (6) at (2,1) [scale=.75, label=left:$h^2_2$] {};
	\vertex (7) at (5, 2) [scale=.75, label=right:$h^3_1$] {};
	\vertex (8) at (6, 1) [scale=.75, label=right:$h^3_2$] {};
	\vertex (9) at (5, 0) [scale=.75, label=right:$h^3_3$] {};
	\vertex (10) at (4, 1) [scale=.75, label=left:$h^3_4$] {};
	\node(A) at (2, -1.6) []{$H$};
	\node(B) at (11.5, -2.2)[]{$J_1$};
	
	\vertex (11) at (11.5, 3.5) [scale=.75, label=above:$h^1_1$] {};
	\vertex (12) at (12, 3) [scale=.75, label=right:$y^1_1$] {};
	\vertex (13) at (13, 2) [scale=.75, label=above:$x^1_2$] {};
	\vertex (14) at (13.5, 1.5) [scale=.75, label=right:$h^1_2$] {};
	\vertex (15) at (13, 1) [scale=.75, label=below:$y^1_2$] {};
	\vertex (16) at (12, 0) [scale=.75, label=right:$x^1_3$] {};
	\vertex (17) at (11.5, -.5) [scale=.75, label=below:$h^1_3$] {};
	\vertex (18) at (11, 0) [scale=.75, label=left:$y^1_3$] {};
	\vertex (19) at (10,1) [scale=.75, label=below:$x^1_4$] {};
	\vertex (20) at (9.5, 1.5) [scale=.75, label=left:$h^1_4$] {};
	\vertex (21) at (10, 2) [scale=.75, label=above:$y^1_4$] {};
	\vertex (22) at (11, 3) [scale=.75, label=left:$x^1_1$] {};

	\path
	(1) edge[line width=1.4pt] (2)
	(2) edge[line width=1.4pt] (3)
	(3) edge[line width=1.4pt] (4)
	(4) edge[line width=1.4pt] (1)
	(5) edge[line width=1.4pt, bend right=30] (6)
	(5) edge[line width=1.4pt, bend left=30] (6)
	(7) edge[line width=1.4pt] (8)
	(8) edge[line width=1.4pt] (9)
	(9) edge[line width=1.4pt] (10)
	(7) edge[line width=1.4pt] (10)
	(2) edge (4)
	(3) edge (5)
	(6) edge (7)
	(1) edge (9)
	(8) edge (10)
	(11) edge (12)
	(12) edge (13)
	(13) edge (14)
	(14) edge (15)
	(15) edge (16)
	(16) edge (17)
	(17) edge (18)
	(18) edge (19)
	(19) edge (20)
	(20) edge (21)
	(21) edge (22)
	(22) edge (11)
	(22) edge (12)
	(13) edge (15)
	(16) edge (18)
	(19) edge (21)
	;
\end{tikzpicture}
\end{center}
\caption{Creating $J_1$ from $C_1'$ in a $2$-factor in the multigraph $H$.}
\label{fig:C_i}
\end{figure}

Suppose first that $G$ is obtained from $H$ by replacing each vertex of $H$ with a triangle, but not replacing any edge of $H$ with a string of diamonds. This means that for each  $i \in [k]$, we can write $C_i' : h_1^ih_2^i\dots h_{m_i}^ih_1^i$. Denote by $J_i$ the induced subgraph of $G$ with Hamiltonian cycle $x_1^ih_1^iy_1^ix_2^ih_2^iy_2^i\dots x_{m_i}^ih_{m_i}^iy_{m_i}^ix_1^i$, where $x_j^ih_j^iy_j^i$ represents replacing the vertex $h_j^i$ in $H$ with the triangle $x_j^ih_j^iy_j^i$. Note that the edge  $y_j^ix_{j+1}^i$ in $G$  represents the edge $h_j^ih_{j+1}^i$  on $C_i'$. (See Fig.~\ref{fig:C_i} which depicts creating $J_1$ from $C_1'$ in a $2$-factor of the underlying multigraph $H$ for $G$). Furthermore, there exists a perfect matching in $G$ that saturates the set $S$ defined by

\[S=\bigcup_{i\in[k]} \{h^i_1, h^i_2, h^i_3, \dots, h^i_{m_i}\}\]
since there exists such a perfect matching $M$ in $H$.   In fact, removing the edges of the $2$-factor from $H$, each vertex of $H$ has degree $1$ in the resulting graph. (For instance, in Fig.~\ref{fig:C_i}, this perfect matching consists of the edges: $h_4^1h_2^1$, $h_1^1h_1^2$, $h_2^2h_1^3$, $h_4^3h_2^3$ and $h_3^1h_3^3$.)
Let $M = \{w_1v_1, \dots , w_tv_t\}$. For each edge $w_iv_i \in M$, assign $w_i$ the color $2_a$ and $v_i$ the color $2_b$. We also assign each vertex of the form $x^i_j$ the color $1_a$ and each vertex of the form $y^i_j$ the color $1_b$. One can easily verify that this is indeed a $(1, 1, 2, 2)$-coloring of $G$.

Next, suppose that $G$ is obtained from $H$ by replacing each vertex of $H$ with a triangle and some edges of $H$ with strings of diamonds. As above, for each $i \in [k]$ we write $C_i': h_1^ih_2^i\dots h_{m_i}^ih_1^i$. Denote by  $J_i$ the subgraph of $G$ induced by $\{x_1^i,h_1^i,y_1^i,x_2^i,h_2^i,y_2^i,\dots, x_{m_i}^i,h_{m_i}^i,y_{m_i}^i\}$, where $x_j^i,h_j^i,y_j^i$ are the vertices of the triangle $x_j^ih_j^iy_j^i$, which replaced the vertex $h_j^i$ in $H$. It is possible that some edges of the form $y^i_jx^i_{j+1}$ are not in $G$, since they were replaced with a string of diamonds, using diamonds $D^{ij}_1, \dots, D^{ij}_{r_{ij}}$. We will refer to this string of diamonds as a {\em Type $1$ string of diamonds}, where the string of diamonds replaced an edge of $J_i$ that corresponds to a cycle of a 2-factor of $H$.
Label the vertices on each $D^{ij}_{\ell}$ for $\ell \in [r_{ij}]$ as $a^{ij}_{\ell}, b^{ij}_{\ell}, c^{ij}_{\ell}, d^{ij}_{\ell}$ where $b^{ij}_{\ell}$ and $c^{ij}_{\ell}$ are the interior vertices of $D^{ij}_{\ell}$, $d^{ij}_{\ell}$ is adjacent to $a^{ij}_{\ell+1}$ in the string of diamonds, and in $G$ $a^{ij}_1$ is adjacent to $y^i_j$ and $d^{ij}_{r_{ij}}$ is adjacent to $x^i_{j+1}$.

As above, there exists a perfect matching $M$ between the vertices of \[S=\bigcup_{i\in[k]} \{h^i_1, h^i_2, h^i_3, \dots, h^i_{m_i}\}.\]
Let $M =  \{w_1v_1, \dots , w_tv_t\}$. Next, we describe a coloring of the vertices of $G$, for which we can verify it is a $(1,1,2,2)$-coloring.

For each $w_iv_i \in M$, color $w_i$ with $2_a$ and $v_i$ with $2_b$. Next, color each vertex of the form $x^i_j$ with $1_a$ and each vertex of the form $y^i_j$ with $1_b$. For each $\ell \in [r_{ij}]$ where $y_j^ix_{j+1}^i$ is replaced with a string of diamonds, color each vertex of the form $b^{ij}_{\ell}$ with $2_a$, each vertex of the form $c^{ij}_{\ell}$  with $2_b$, each vertex of the form $d^{ij}_{\ell}$ with $1_b$, and each vertex of the form $a^{ij}_{\ell}$ with $1_a$.  Finally, suppose some edge of $M$, say $w_iv_i$, is replaced with a string of diamonds using diamonds $D^i_1, \dots, D^i_{s_i}$. We will refer to such a string of diamonds as a {\em Type $2$ string of diamonds}. Label the vertices on each $D^i_{\ell}$ for $\ell \in [s_i]$ as $a^i_{\ell}, b^i_{\ell}, c^i_{\ell}, d^i_{\ell}$ where $b^i_{\ell}$ and $c^i_{\ell}$ are the interior vertices of $D^i_{\ell}$, $d^i_{\ell}$ is adjacent to $a^i_{\ell +1}$ in the string of diamonds, and in $G$ $w_i$ is adjacent to $a^i_1$ and $v_i$ is adjacent to $d^i_{s_i}$. If $w_i$ is assigned the color $2_a$, then assign each vertex of the form $a^i_{\ell}$ the color $2_b$ and each vertex of the form $d^i_{\ell}$ the color $2_a$. On the other hand, if $w_i$ is assigned the color $2_b$, then assign each vertex of the form $a^i_{\ell}$ the color $2_a$ and each vertex of the form $d^i_{\ell}$ the color $2_b$. Regardless of the color assigned to $w_i$, assign each vertex of the form $b^i_{\ell}$ the color $1_a$ and each vertex of the form $c^i_{\ell}$ the color $1_b$. (See also Fig~\ref{fig:coloring} which depicts how such a $(1, 1, 2, 2)$-coloring is created based on a chosen $2$-factor from the underlying multigraph $H$ of $G$, where vertices representing colors $2_a$ and $2_b$ are enlarged, black vertices represent color $1_a$, and white vertices represent color $1_b$.) It is straightforward to verify that the resulting coloring is indeed a $(1,1,2,2)$-coloring of $G$.
\end{proof}

\medskip

A $(1,1,2,2)$-coloring of a $2$-connected, claw-free cubic graph $G$, presented in the proof of Theorem~\ref{thm:2connected}, will be called a {\em canonical coloring} of $G$. We will refer to it in the proof of the result, where we omit the $2$-connectedness condition. See Fig.~\ref{fig:coloring} presenting such a coloring.

  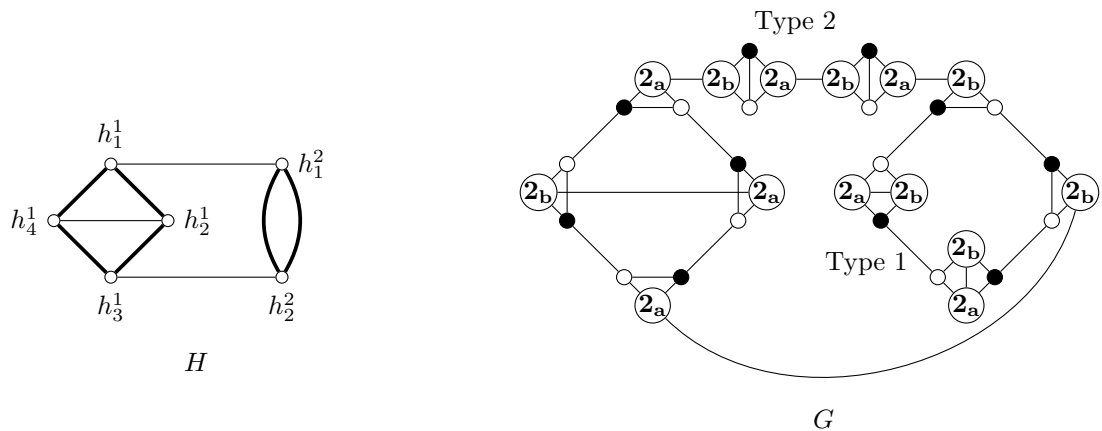
\begin{figure}[h]
\begin{center}
\begin{tikzpicture}[scale=.75]
\tikzstyle{vertex}=[circle, draw, inner sep=0pt, minimum size=6pt]
\tikzset{vertexStyle/.append style={rectangle}}
	\vertex (1) at (0,0) [scale=.75, label=below:$h_3^1$] {};
	\vertex (2) at (-1, 1) [ scale=.75, label=left:$h_4^1$] {};
	\vertex (3) at (0,2) [scale=.75, label=above:$h_1^1$] {};
	\vertex (4) at (1,1) [scale=.75, label=right:$h^1_2$] {};
	\vertex (5) at (3, 2) [scale=.75, label=right:$h^2_1$] {};
	\vertex (6) at (3,0) [scale=.75, label=below:$h^2_2$] {};
	\node(A) at (1.5, -1.5) []{$H$};
	\node(B) at (12.5, -2.5)[]{$G$};
	\node(C) at (12, 4.5)[]{Type 2};
	\node(D) at (13.25, 0.25)[]{Type 1};

	\vertex (11) at (9.5, 3.5) {$\mathbf{2_a}$};
	\vertex (12) at (10, 3) [label=right:$$] {};
	\vertex (13) at (11, 2) [scale=.95, fill=black,label=above:$$] {};
	\vertex (14) at (11.5, 1.5) [] {$\mathbf{2_a}$};
	\vertex (15) at (11, 1) [scale=.95, label=below:$$] {};
	\vertex (16) at (10, 0) [scale=.95,fill=black, label=right:$$] {};
	\vertex (17) at (9.5, -.5)  {$\mathbf{2_a}$};
	\vertex (18) at (9, 0) [scale=.95, label=left:$$] {};
	\vertex (19) at (8,1) [scale=.95,fill=black, label=below:$$] {};
	\vertex (20) at (7.5, 1.5)  {$\mathbf{2_b}$};
	\vertex (21) at (8, 2) [scale=.95, label=above:$$] {};
	\vertex (22) at (9, 3) [scale=.95, fill=black,label=left:$$] {};
	
	\vertex (23) at (10.7, 3.5) {$\mathbf{2_b}$};
	\vertex (24) at (11.2, 4) [scale=.95, fill=black,label=left:$$] {};
	\vertex (25) at (11.2, 3) [scale=.95, label=left:$$] {};
	\vertex (26) at (11.7, 3.5) [] {$\mathbf{2_a}$};
	\vertex (27) at (12.8, 3.5) [] {$\mathbf{2_b}$};
	\vertex (28) at (13.3, 4) [scale=.95, fill=black,label=left:$$] {};
	\vertex (29) at (13.3, 3) [scale=.95, label=left:$$] {};
	\vertex (30) at (13.8, 3.5) [] {$\mathbf{2_a}$};
	\vertex (31) at (15, 3.5)  {$\mathbf{2_b}$};
	\vertex (32) at (15.5, 3) [scale=.95, label=left:$$] {};
	\vertex (33) at (16.5, 2) [scale=.95, fill=black, label=left:$$] {};
	\vertex (34) at (17, 1.5)  {$\mathbf{2_b}$};
	\vertex (35) at (16.5, 1) [scale=.95, label=left:$$] {};
	\vertex (36) at (15.5, 0) [scale=.95, fill=black,label=left:$$] {};
	\vertex (37) at (15, .5) [] {$\mathbf{2_b}$};
	\vertex (38) at (15, -.5) [] {$\mathbf{2_a}$};
	\vertex (39) at (14.5, 0) [scale=.95, label=left:$$] {};
	\vertex (40) at (13.5, 1) [scale=.95,fill=black, label=left:$$] {};
	\vertex (41) at (13, 1.5) [] {$\mathbf{2_a}$};
	\vertex (42) at (14, 1.5) [] {$\mathbf{2_b}$};
	\vertex (43) at (13.5, 2) [scale=.95, label=left:$$] {};
	\vertex (44) at (14.5, 3) [scale=.95, fill=black,label=left:$$] {};

	\path
	(1) edge[line width=1.4pt] (2)
	(2) edge[line width=1.4pt] (3)
	(3) edge[line width=1.4pt] (4)
	(4) edge[line width=1.4pt] (1)
	(1) edge (6)
	(5) edge[line width=1.4pt,bend right=30] (6)
	(5) edge[line width=1.4pt,bend left=30] (6)
	(2) edge (4)
	(3) edge (5)
	(11) edge (12)
	(12) edge (13)
	(13) edge (14)
	(14) edge (15)
	(15) edge (16)
	(16) edge (17)
	(17) edge (18)
	(18) edge (19)
	(19) edge (20)
	(20) edge (21)
	(21) edge (22)
	(22) edge (11)
	(22) edge (12)
	(13) edge (15)
	(16) edge (18)
	(19) edge (21)
	(11) edge (23)
	(23) edge (24)
	(23) edge (25)
	(24) edge (25)
	(24) edge (26)
	(25) edge (26)
	(26) edge (27)
	(27) edge (28)
	(27) edge (29)
	(28) edge (29)
	(28) edge (30)
	(29) edge (30)
	(30) edge (31)
	(31) edge (32)
	(32) edge (33)
	(33) edge (34)
	(34) edge (35)
	(35) edge (36)
	(36) edge (37)
	(37) edge (38)
	(38) edge (39)
	(39) edge (40)
	(40) edge (41)
	(41) edge (42)
	(42) edge (43)
	(43) edge (44)
	(44) edge (31)
	(32) edge (44)
	(33) edge (35)
	(36) edge (38)
	(37) edge (39)
	(40) edge (42)
	(41) edge (43)
	(14) edge (20)
	(17) edge[bend right=60] (34)
	
	;
\end{tikzpicture}
\end{center}
\caption{Creating a $(1, 1, 2, 2)$-coloring for $G$ based on the $2$-factor chosen in the underlying multigraph $H$.}
\label{fig:coloring}
\end{figure}

\medskip
From the construction in the proof of Theorem~\ref{thm:2connected}, we infer the following property of a canonical coloring, which will be used in the rest of the section.
\begin{observation}
\label{obs:2}
Let $G\ne K_4$ be a $2$-edge-connected, claw-free cubic graph, which is not a ring of diamonds, and let $c$ be a canonical $(1, 1, 2, 2)$-coloring of $G$. If $u$ is assigned the color $1_a$ (resp. $1_b$), then either $u$ has two neighbors with color $1_b$ (resp. $1_a$), or $u$ is on a diamond.
\end{observation}

\vskip1mm

Our next goal is to extend our $(1, 1, 2, 2)$-coloring to a claw-free cubic graph with bridges. Suppose $G$ is a graph with non-empty bridge set $B(G)$. For each edge  $e$ in  $B(G)$, there exist distinct components $G_i$ and $G_j$ of $G - B(G)$ such that $e$ is incident to some vertex in $V(G_i)$ and some vertex in $V(G_j)$. In such a case, we will say that {\it $e$ is incident to  $G_i$ and $G_j$}.

Let $G$ be a connected, claw-free cubic graph with $\vert B(G) \vert = b \geq 0$,  and let $G_0, G_1, \dots, G_b$ be the components of $G - B(G)$.   Let $T_G$ denote the simple graph with vertex set $V(T_G) = \{g_0, g_1, g_2, \dots, g_b\}$ and edge set  $E(T_G)= \{g_ig_j \mid $ some edge $e \in B(G)$ is incident to $G_i$ and $G_j\}$. We observe that $T_G$ is a tree, and that  $G$ has no bridges if and only if  $T_G$ is isomorphic to $K_1$.
Notice first that if $0 \le i \le b$, then no vertex in $G_i$ is a leaf.
Indeed, suppose $v$ is a leaf in $G_i$ that is adjacent to its support vertex $w$ in $G_i$.  The other two edges incident to $v$, say $e_j=vu_j$ and $e_k=vu_k$, belong to $B(G)$.  This implies that $\{v,w,u_j,u_k\}$ induces a claw in $G$, which is a contradiction.
Thus, if $G_i\ne K_1$, then $\delta(G_i) \ge 2$ which implies that $n(G_i) \ge 3$. We may also conclude that
$G_i$ contains no bridges. Indeed, since $T_G$ is a tree, if $uv$ is a bridge in $G_i$, there does not exist a path from the component of $G_i-uv$ that contains $u$ to the component of $G_i-uv$ that contains $v$ except for the trivial path $uv$. Since $uv \not\in B(G)$, there exists another $uv$-path $P$ in $G$ that is not a path in $G_i$, meaning $P$ contains a bridge $xy\in B(G)$. However, this implies that $G-xy$ contains an $xy$-path, contradicting the assumption that $xy$ is a bridge.
Therefore, we may classify each component $G_i$ as follows:
\vskip 5pt
\indent Type I:  isomorphic to $K_1$, or
\vskip 5pt
\indent Type II:  isomorphic to the $m$-cycle $C_m$ for some $m \geq 3$, or
\vskip 5pt
\indent Type III: isomorphic to some 2-edge-connected graph with maximum degree $3$.
\vskip 5pt

\noindent We note that if $G_i$ is of Type I or Type II, then $b > 0$ and $g_i$ is an interior vertex of $T_G$. Otherwise, if $G_i$ is of Type III, then  $g_i$ is a leaf of $T_G$ if and only if $G_i$ has precisely one vertex of degree 2.   To illustrate, we observe that for the graph $G$ of Figure~\ref{fig:exampleGi}, $T_{G}$ is isomorphic to $K_{1,3}$ and $G - B(G)$ has one component of Type II and three components of Type III.

Since we are only dealing with claw-free graphs, the following observation will also be useful.

\begin{observation}
\label{obs:1}
If $G$ is a claw-free graph, then no $G_i$ is of Type I, and if $G_i$ is of Type II, then $G_i = K_3$.
\end{observation}

\smallskip

\begin{figure}[h]
\begin{center}
\begin{tikzpicture}[scale=.75]
\tikzstyle{vertex}=[circle, draw, inner sep=0pt, minimum size=6pt]
\tikzset{vertexStyle/.append style={rectangle}}
	\vertex (1) at (0.75,0.75) [fill, scale=.75] {};
	\vertex (2) at (2.25,0.75) [fill, scale=.75] {};
	\vertex (3) at (1.5,1.5) [fill, scale=.75]{};
	\vertex (4) at (1.5, 2.5) [fill, scale=.75]{};
	\vertex (5) at (2.25, 3.5) [fill, scale=.75]{};
	\vertex (6) at (.75, 3.5) [fill, scale=.75]{};
	\vertex (7) at (-.25, .75) [fill, scale=.75]{};
	\vertex (8) at (-1.25, 1.5) [fill, scale=.75]{};
	\vertex (9) at (-1.25, 0) [fill, scale=.75]{};
	\vertex (10) at (3.25, .75) [fill, scale=.75]{};
	\vertex (11) at (4.25, 1.5) [fill, scale=.75]{};
	\vertex (12) at (4.25, 0) [fill, scale=.75]{};

    \vertex (a) at (2.25,4.5) [fill, scale=.75]{};
    \vertex (b) at (.75,4.5) [fill, scale=.75]{};
    \vertex (c) at (-2.25,0) [fill, scale=.75]{};
    \vertex (d) at (-2.25,1.5) [fill, scale=.75]{};
    \vertex (r) at (5.25,1.5) [fill, scale=.75]{};
    \vertex (s) at (5.25,0) [fill, scale=.75]{};
    \vertex (f) at (-2.25,.75) [fill, scale=.75]{};
    \vertex (g) at (-3.25,.75) [fill, scale=.75]{};
    \vertex (t) at (5.25,.75) [fill, scale=.75]{};
    \vertex (u) at (6.25,.75) [fill, scale=.75]{};
    \vertex (p) at (.75,5.5) [fill, scale=.75]{};
    \vertex (q) at (2.25,5.5) [fill, scale=.75]{};
	\path
		(1) edge (2)
		(2) edge (3)
		(1) edge (3)
		(3) edge (4)
		(4) edge (5)
		(4) edge (6)
        (5) edge (6)
        (a) edge (p)
        (a) edge (q)
        (p) edge (q)
		(6) edge (b)
        (5) edge (a)
        (b) edge (p)
        (b) edge (q)
		(1) edge (7)
        (c) edge (d)
		(7) edge (8)
		(7) edge (9)
		(8) edge (d)
		(8) edge (9)
        (9) edge (c)
		(c) edge (g)
        (c) edge (f)
        (f) edge (g)
        (d) edge (f)
        (d) edge (g)
		(2) edge (10)
        (r) edge (s)
		(10) edge (11)
		(10) edge (12)
		(11) edge (r)
		(11) edge (12)
        (12) edge (s)
		(s) edge (t)
        (s) edge (u)
        (r) edge (t)
        (r) edge (u)
        (t) edge (u)
	;
\end{tikzpicture}
\end{center}
\caption{ The graph $G$ with $T_G=K_{1,3}$.} \label{fig:exampleGi}
\end{figure}
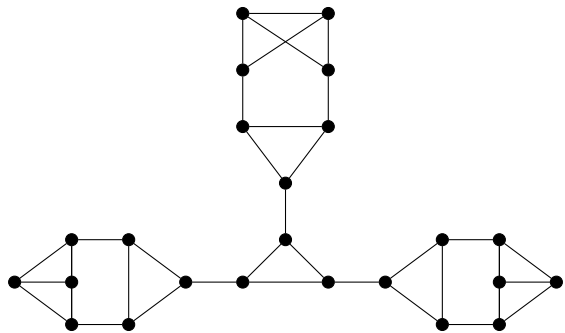

In what follows, we will consider the components $G_0, \dots, G_b$ of $G- B(G)$, which are $2$-edge-connected , but unfortunately not cubic. Note that if such a component $G_i$ has fewer than $5$ vertices, then $G_i$ is isomorphic to either $K_3$ or a diamond. These two cases turn out to be easy and we will present appropriate colorings separately. If the order of $G_i$ is at least $5$, then we will create a $2$-edge-connected, claw-free cubic graph $\widetilde{G_i}$ from $G_i$, and use the same coloring strategy for $\widetilde{G_i}$ as that given in the proof of Theorem~\ref{thm:2connected}. That is, we will use a canonical coloring for every such graph $\widetilde{G_i}$.

Now we are ready to prove that any claw-free cubic graph is $(1, 1, 2, 2)$-colorable.

\begin{theorem}\label{thm:main}
If $G$ is a claw-free cubic graph, then $G$ is $(1, 1, 2, 2)$-colorable.
\end{theorem}

\begin{proof} We may assume that $G$ is connected.  We let $B(G)$ represent the set of bridges in $G$ and let $G_0, G_1, \dots, G_b$ be the components of $G - B(G)$. If $b=0$, then the conclusion follows by Theorem~\ref{thm:2connected}. We consider the tree $T_G$ with $V(T_G) = \{g_0, g_1, \dots, g_b\}$ where $g_ig_j \in E(T_G)$ if and only if there exists $e \in B(G)$ incident to $G_i$ and $G_j$. Choose a path $P$ of $T_G$ of length $d={\rm{diam}}(T_G)$ and let $g_0$ be a leaf of $P$. We root $T_G$ at $g_0$ and define $A_i = \{v \in V(T_G) \mid d_{T_G}(g_0, v) = i\}$. Clearly, $A_0 = \{g_0\}$. Note that if $i>0$, then each vertex in $A_i$ has exactly one neighbor in $A_{i-1}$. It follows that if $g_j \in A_i$ represents component $G_j$, then there exists a unique vertex $p \in V(G_j)$ with degree $2$ in $G_j$ such that $p$ has a neighbor $q\in V(G_k)$ where $g_k$ is the parent of $g_j$ in $T_G$. In such a case, we call $q$ the ``up-neighbor" of $p$. We also know that every vertex of $A_d$ is a leaf in $T_G$. Reindexing if necessary, we may assume that $g_0$ represents $G_0$.

First, consider $g_0$, which is a leaf in $T_G$. Note that $G_0$ is of Type III and there exists only one vertex of degree $2$ in $G_0$, call it $v$. Let $u$ and $w$ be the neighbors of $v$ in $G_0$. Note that $u$ and $w$ are adjacent in $G$ since $G$ is claw-free. Let $s$ be the neighbor of $u$ different from $v$ and $w$, and let $y$ be the neighbor of $w$ different from $u$ and $v$.  Note that $s\ne y$ for otherwise $G_0$ is a diamond with two vertices of degree $2$, which is a contradiction.

We claim that  $s$ and $y$ are not adjacent in $G$. If $sy\in E(G)$, then $s$ and $y$ have a common neighbor $z$ as $G$ is claw-free. However, this implies that either $z$ has degree $2$ which is impossible, or $z$ is incident to a bridge which is another contradiction. Thus, $s$ and $y$ are not adjacent in $G$. Consider the graph $\widetilde{G_0}$ obtained from $G_0$ by removing the vertices $u, v$, and $w$, and adding the edge $sy$.  We claim that $sy$ is not on a triangle of $\widetilde{G_0}$ unless $\widetilde{G_0} = K_4$. Suppose to the contrary that $syz$ is a triangle in $\widetilde{G_0}$. Since $\{u, z\}$ is independent in $G$ and $G$ is claw-free, the remaining neighbor of $s$ in $G$, call it $a$, is adjacent to $z$. Similarly, the remaining neighbor of $y$ in $G$, call it $d$, is adjacent to $z$. If $a\ne d$, then $\deg_G(z) = 4$, which is a contradiction. Thus, $a=d$ and $\widetilde{G_0}= K_4$. In this case,  we assign $s$ and $w$ the color $1_b$, assign $u$ and $y$ the color $1_a$, assign $v$ and $a$ the color $2_a$ and assign $z$ the color $2_b$.

Therefore, we may assume $sy$ is not on a triangle of $\widetilde{G_0}$ and  $\widetilde{G_0} \ne K_4$.  Let $H$ be the underlying multigraph of $\widetilde{G_0}$. Since $sy$ is not on a triangle of $\widetilde{G_0}$, it is not created from $H$ by replacing a vertex of $H$ with a triangle nor is it on a diamond in $\widetilde{G_0}$. Therefore, we may associate $sy$ with an edge $e$ in $H$ that may or may not be replaced with a string of diamonds. By Theorem~\ref{thm:multigaph-1factor} there exists a $2$-factor of $H$ containing $e$ meaning that when we create $\widetilde{G_0}$ from $H$ as in the proof of Theorem~\ref{thm:2connected}, $sy$ is either of the form $y_j^ix_{j+1}^i$ in which case $s$ and $y$ receive colors from $\{1_a, 1_b\}$, or $sy$ is incident to at least one diamond in a string of diamonds of Type 1 that replaces $e$. Recall that by the canonical coloring (from the proof of Theorem~\ref{thm:2connected}), also in this case $s$ and $y$ are given colors from $\{1_a, 1_b\}$. Without loss of generality,  we may assume $s$ receives $1_a$ and $y$ receives $1_b$. We then color each vertex of $G_0$ such that vertices of $G_0$ that are in $\widetilde{G_0}$ receive the same colors as assigned in the described coloring of $\widetilde{G_0}$, $u$ receives color $1_b$, $w$ receives color $1_a$, and $v$ receives color $2_a$. Therefore, $G_0$ is assigned a $(1, 1, 2, 2)$-coloring, where $v$ received color $2_a$.

Now we keep working our way down the tree by coloring the vertices of the subgraphs $G_i$ in the BFS order with respect to $T_G$. That is, we will next assign colors to the vertices of $G_i$ for the unique $G_i$ whose vertex in $T_G$ is in $A_1$. Then we proceed to assign colors to the vertices of $G_i$ for each $G_i$ whose vertex in $T_G$ is in $A_2$, etc. For each $G_i$ which is a $K_3$ or a diamond, we let $x_1^i$ be the only vertex that has an up-neighbor. If $G_i = K_3$,  we let $u_i$ and $v_i$ be the other remaining vertices and we immediately assign $u_i$ the color $1_a$ and $v_i$ the color $1_b$. (Due to this, $x_1^i$ will be assigned a color in $\{2_a,2_b\}$.) Similarly,  if $G_i$ is a diamond, where $u_i$ and $v_i$ are the interior vertices of the diamond, then we immediately assign $u_i$ the color $1_a$ and $v_i$ the color $1_b$. (Hence, the colors of $x_1^i$ and $x_2^i$ will be different colors from the set $\{2_a,2_b\}$.)

For each $G_i$ of Type III, where $|V(G_i)|\ge 5$, we let $X_i = \{x_1^i, x_2^i, \dots, x_{r_i}^i\}$ be all the vertices in $G_i$ of degree $2$ such that $x_1^i$ is the only vertex that has an up-neighbor. In order to assign colors to the vertices of $G_i$, we will create a $2$-edge-connected, claw-free cubic graph $\widetilde{G_i}$ in one of two ways. Since $G_i$ has order at least $5$ and $G$ is claw-free, it follows  that the set $X_i$ is independent in $G$.  First, if $r_i$ is even, then we add the edges $x^i_jx^i_{j+1}$ for $1 \le j \le r_i-1$ where $j$ is odd. It is clear that $\widetilde{G_i}$ is cubic and we know $\widetilde{G_i}$ is $2$-edged-connected as $G_i$ is $2$-edge-connected. To see that $\widetilde{G_i}$ is claw-free, note that any vertex $v$ in $G_i$ with degree $3$ was already assumed to not be the center of a claw in $G$ and any vertex $v$ of degree $2$ in $G_i$ is incident to a bridge so the two neighbors of $v$ in $G_i$ must be adjacent. Therefore, $\widetilde{G_i}$ is indeed claw-free.

 Next, if $r_i$ is odd, then first we let the two neighbors of $x_1^i$ in $G_i$ be $u_i$ and $w_i$. As in the case when we considered $G_0$, note that $u_i$ and $w_i$ are adjacent in $G$ as $G$ is claw-free. Let $s_i$ be the neighbor of $u_i$ different from $x_1^i$ and $w_i$, and let $y_i$ be the neighbor of $w_i$ different from $u_i$ and $x_1^i$. We may assume that $s_i\ne y_i$ for otherwise $G_i$ is a diamond which implies $r_i=2$, which is a contradiction. Note that we may assume $s_i$ and $y_i$ are not adjacent in $G$ for otherwise $s_i$ and $y_i$ have a common neighbor $z_i$ as $G$ is claw-free. However, this implies $z_i$ has degree $2$ implying $r_i=2$, which cannot be, or $z_i$ is incident to a bridge, another contradiction. Thus, $s_i$ and $y_i$ are not adjacent in $G$. We let $\widetilde{G_i}$ be the graph obtained from $G_i - \{x_1^i, u_i, w_i\}$ by adding the edge $s_iy_i$ and every edge of the form $x_j^ix_{j+1}^i$ where $2 \le j \le r_i-1$ where $j$ is even. It is clear that $\widetilde{G_i}$ is cubic and $2$-edge-connected as $G_i$ is assumed to be $2$-edge-connected.  To see that $\widetilde{G_i}$ is claw-free, note that any vertex $v$ in $G_i$ with degree $3$ was already assumed to not be the center of a claw in $G$ and any vertex $v$ of degree $2$ in $G_i$ is incident to a bridge so the two neighbors of $v$ in $G_i$ must be adjacent. Therefore, $\widetilde{G_i}$ is indeed claw-free.

 \vskip2mm
 \noindent\textbf{Claim 1} For each $1 \le i \le b$ where $G_i$ is of Type III, the up-neighbor of $x_1^i$ in $G_j$ is not on a diamond in $\widetilde{G_j}$ unless $G_j$ is itself a diamond.
 \vskip2mm

 \noindent\textit{Proof} Suppose to the contrary that the up-neighbor $w$ of $x_1^i$ is in $G_j$, where $G_j$ is not a diamond, and yet $w$ is on a diamond $D$ in $\widetilde{G_j}$. Suppose first that $j=0$. It follows that $w$ is an exterior vertex of $D$ as $w$ is the only vertex of degree $2$ in $G_0$. Let $u, v$, and $z$ be the remaining vertices on $D$ where $z$ is an exterior vertex of $D$. As $z$ has degree $3$ in $G_0$, $z$ has a neighbor $z'$ in $G_0$ different from $u$ and $v$. This implies that $zz'$ is a bridge in $G_0$, contradicting the fact that $G_0$ is $2$-edge-connected. Thus, we may assume that $j\ne 0$. Since $G_j$ corresponds to an interior vertex of $T_G$, we know $\{x_1^j, \dots, x_{r_j}^j\}$ contains at least two vertices.  Without loss of generality, assume $w = x_{r_j}^j$ and assume $x_{r_j}^j$ is adjacent to $x_{r_j-1}^j$ in $\widetilde{G_j}$.

 Suppose first that $x_{r_j}^j$ is an interior vertex of the diamond $D$. It follows that $x_{r_j-1}^j$ is also on $D$ as only two neighbors of $x_{r_j}^j$ in $G$ are in $G_j$. Let $u$ and $v$ be the other two vertices on $D$.
If $x_{r_j-1}^j$ is the other interior vertex of $D$, then $\{x_1^i, u, v\}$ is an independent set contradicting the fact that $G$ is claw-free. On the other hand, if $x_{r_j-1}^j$ is an exterior vertex of $D$ and $u$ is the other interior vertex of $D$, then $N_G[x_{r_j-1}^j]$ induces a claw, another contradiction. Therefore, we may assume that $x_{r_j}^j$ is an exterior vertex of the diamond. Moreover, by interchanging the roles of $x_{r_j}^j$ and $x_{r_j-1}^j$ in the above, we may assume that $x_{r_j-1}^j$ is not on the diamond.  Indeed, if $x_{r_j-1}^j$ is on $D$, then it must be an interior vertex of $D$, and in the same way as above we conclude that $N_G[x_{r_j}^j]$ induces a claw which is impossible.
 Label the remaining vertices of $D$ as $u, v$, and $z$ where $z$ is the other exterior vertex of the diamond. Then $z$ is incident to a bridge in $G_j$ as there does not exist a path from $u$ to any vertex in $V(G) - \{u, v, z, x_{r_j}^j\}$ that does not contain either $x_{r_j}^j$ or $z$. However, this is a contradiction since $G_j$ is assumed to be $2$-edge-connected. Therefore, this case cannot occur either and the claim is proved.
 \smallqed
 \vskip2mm

Now, we are in a position to present the $(1,1,2,2)$-coloring of $G_i$ under the assumption that the vertices of $G_j$, where $g_j$ is the parent of $g_i$ in $T_G$, have already been colored. We first deal with the unique vertex $x_1^i$ of $G_i$ having an up-neighbor, and prove that it can receive a color in $\{2_a,2_b\}$.

First, consider two small cases of $G_j$. If $G_j = K_3$, then $x_1^j$ has been assigned color $2_a$ or $2_b$, and its neighbors received colors from $\{1_a,1_b\}$, which implies that $x_1^i$ can receive a color from $\{2_a,2_b\}$. Similarly, if $G_j$ is a diamond, then its interior vertices were assigned colors $1_a$ and $1_b$, while $x_2^j$ was assigned color $2_a$ or $2_b$, again implying that $x_1^i$ can receive a color from $\{2_a,2_b\}$. In the next claim, we may thus assume that $G_j$ is of Type III and $|V(G_j)|\ge 5$.

\vskip2mm
\noindent\textbf{Claim 2} In a canonical $(1,1,2,2)$-coloring of $\widetilde{G_j}$, where $|V(G_j)|\ge 5$, there is a color from $\{2_a,2_b\}$ such that no vertex in $N_{G_j}[x_t^j]$, where $t\in\{2,\ldots, r_j\}$, receives that color.
\vskip2mm
\noindent\textit{Proof}
Clearly, since $|V(G_j)|\ge 5$, we infer that $G_j$ is of Type III. Let $w \in \{x_2^j, \dots, x_{r_j}^j\}$  be the neighbor of $x_1^i$ in $G_j$, and without loss of generality, we may assume that $w = x_{\ell}^j$. Note first that if $w \in V(G_0)$, then we have already assigned $w$ the color $2_a$, and its neighbors received colors from $\{1_a,1_b\}$, so the claim holds.

Thus, let $j>0$, and consider a canonical $(1,1,2,2)$-coloring of $\widetilde{G_j}$. We may assume that $x_{\ell}^j$ is adjacent to $x_k^j$   in $\widetilde{G_j}$.
Suppose first that $x_{\ell}^j$ is assigned either color $1_a$ or $1_b$.  Then by Observation~\ref{obs:2} and Claim 1, $x_{\ell}^j$ has two neighbors in $\widetilde{G_j}$ both of which received color $1_a$ or both received color $1_b$, and exactly one neighbor, say $z$, received a color in $\{2_a,2_b\}$.  Thus the claim holds in this case.

Next, suppose that $x_{\ell}^j$ is assigned either color $2_a$ or $2_b$. We may assume without loss of generality that it is assigned the color $2_a$. If $x_k^j$ is assigned the color $2_b$, then the other two neighbors of $x_{\ell}^j$ in $\widetilde{G_j}$ receive colors in $\{1_a,1_b\}$. However, since $x_k^j\notin N_{G_j}[x_\ell^j]$, again the claim holds.

So we shall assume that $x_k^j$ is assigned either color $1_a$ or $1_b$. Let $u$ and $v$ be the other neighbors of $x_{\ell}^j$ in $G_j$. Recall that $u$ and $v$ must be adjacent, since otherwise $G$ would have a claw with $x_{\ell}^j$ as its center.  We may assume, without loss of generality, that $v$ is assigned the color $2_b$ (for if $u$ and $v$ both received colors in $\{1_a,1_b\}$, the claim already holds). Note that in a canonical coloring, the only edges $wz$ where $w$ is assigned the color $2_a$ and $z$ the color $2_b$ are of the following types. Either
\begin{enumerate}
       \item[(a)] $w$ and $z$ are interior vertices of a diamond on a Type 1 string of diamonds that replaced an edge in a $2$-factor of $H$, or
    \item[(b)] $wz$ is incident to two different diamonds in a Type 2 string of diamonds that replaced an edge in a $2$-factor of $H$, or
    \item[(c)] $w = h_m$ and $z$ is either $h_n$ or an exterior vertex of a diamond on a Type 2 string of diamonds that replaced the edge $h_mh_n$.
\end{enumerate}
   Since $G_j$ is not a diamond, we infer by Claim 1 that $x_{\ell}^j$ does not lie on a diamond in $\widetilde{G_j}$, and so case (a) does not happen for $x_{\ell}^j$ and $v$. In either of the cases (b) and (c), the edge $wz$ does not lie on a triangle, however, $vx_\ell^j$ lies on the
   triangle $uvx_\ell^j$ in $\widetilde{G_j}$. Since none of the above cases is possible, we infer that $u$ and $v$ both received colors in $\{1_a,1_b\}$, and so the statement of the claim follows.
 \smallqed

\bigskip
Note that by Claim 2 and by the arguments preceding Claim 2, where $G_j=K_3$ or $G_j$ is a diamond, there is a color from $\{2_a,2_b\}$ such that all vertices in $G_j$ that are at distance at most $2$ from $x_1^i$, are not assigned that color. Therefore this color can and will be assigned to the vertex $x_1^i$. Without loss of generality, we may assume that this color is $2_a$.

\vskip2mm
\noindent\textbf{Claim 3} There exists a $(1,1,2,2)$-coloring of $G_i$ so that $x_1^i$ is assigned the color $2_a$.
\vskip2mm
\noindent\textit{Proof} If $G_i = K_3$, then certainly this is true as we have already assigned $u_i$ the color $1_a$ and $v_i$ the color $1_b$. If $G_i$ is a diamond, then $x_1^i$ and $x_2^i$ are the exterior vertices and we have already assigned the interior vertices $u_i$ and $v_i$ colors $1_a$ and $1_b$ respectively. In addition, by the assumption preceding this claim, $x_2^i$ receives color $2_b$, while $x_1^i$ receives color $2_a$.

Thus, we shall assume that $G_i$ is of Type III and order at least $5$. Suppose first that $r_i$ is even. We claim  that $x_1^ix_2^i$ is not on a triangle in $\widetilde{G_i}$. Indeed, suppose that $v$ is a common neighbor of $x_1^i$ and $x_2^i$ in $\widetilde{G_i}$, and let $N_{\widetilde{G_i}}(x_1^i)=\{x_2^i,v,a\}$, and $N_{\widetilde{G_i}}(x_2^i)=\{   x_1^i,v,b\}$. Since $G$ is claw-free, both $a$ and $b$ must be adjacent to $v$. Since $v$ has only three neighbors, we infer that $a=b$. That is, in $\widetilde{G_i}$, vertices $a$ and $v$ are common neighbors of $x_1^i$ and $x_2^i$, and are adjacent, which implies $\widetilde{G_i}=K_4$, a contradiction with  $|V(G_i)|\ge5$.

Let $H$ be the underlying multigraph of $\widetilde{G_i}$. Since $x_1^ix_2^i$ is not on a triangle, we can associate $x_1^ix_2^i$ with an edge $h_mh_n \in E(H)$ that may or may not be replaced with a string of diamonds to create $\widetilde{G_i}$. Moreover, by Theorem~\ref{thm:multigaph-1factor}, there exists a perfect matching of $H$ that contains $h_mh_n$ implying that we can create a canonical $(1, 1, 2, 2)$-coloring of $\widetilde{G_i}$ where $x_1^i$ receives color $2_a$ and $x_2^i$ receives color $2_b$. Now, for each vertex $v$ in $G_i$  we assign $v$ the same color as $v$ was assigned in $\widetilde{G_i}$. Since $x_2^i$ is not adjacent to $x_1^i$ in $G$, the distance in $G$ from $x_2^i$ to vertices in $G_j$ is greater than $2$. In addition, the distance in $G$ between every vertex in $V(G_i)\setminus\{x_1^i\}$ colored by color $2_a$ or $2_b$ to vertices in $G_j$ is greater than $2$. This proves the statement of the claim in the case $r_i$ is even.

Finally, assume that $r_i$ is odd. By using the same notation and  the same argument as in the paragraph preceding Claim 1, we let the two neighbors of $x_1^i$ in $G_i$ be $u_i$ and $w_i$, and we note that $u_i$ and $w_i$ are adjacent in $G$. Furthermore, $s_i$ is the neighbor of $u_i$ different from $x_1^i$ and $w_i$, and $y_i$ is the neighbor of $w_i$ different from $u_i$ and $x_1^i$. Recall that $s_i\ne y_i$, and $s_i$ and $y_i$ are not adjacent in $G$. Now, in $\widetilde{G_i}$, $s_i$ and $y_i$ are made adjacent. We claim that $s_iy_i$ is not on a triangle in $\widetilde{G_i}$ unless $\widetilde{G_i} = K_4$. Indeed, suppose that $s_iy_i$ lies on a triangle in $\widetilde{G_i}$, say $as_iy_i$.  Let $b$ (resp. $c$) be the other neighbor of $s_i$ (resp. $y_i$) in $\widetilde{G_i}$.  Since neither $s_i$ nor $y_i$ is the center of a claw in $G$, we see that $ab$ and $ac$ are both edges in $G_i$.  This implies that $b=c$ since $a$ has degree $3$, and so $\widetilde{G_i} = K_4$. (In particular, we also infer that $r_i=1$, and $G_i$ corresponds to a leaf of $T_G$.)  In this case, we assign color $2_a$ to $x_1^i$ and to $a$, color $2_b$ to $b$, color $1_b$ to $u_i$ and $y_i$, and color $1_a$ to $w_i$ and $s_i$. Thus, we may now assume that  $s_iy_i$ is not on a triangle in $\widetilde{G_i}$.

Let $H$ be the underlying multigraph of $\widetilde{G_i}$. It follows that we can associate $s_iy_i$ with an edge $h_mh_n \in E(H)$ that may or may not be replaced with a string of diamonds to create $\widetilde{G_i}$. By Theorem~\ref{thm:multigaph-1factor}, there exists a $2$-factor of $H$ containing $h_mh_n$. Thus, in $\widetilde{G_i}$ the edge $s_iy_i$ is either of the form $y_j^ix_{j+1}^i$ in which case $s_i$ and $y_i$ receive different colors from $\{1_a,1_b\}$, or $s_iy_i$ is incident to at least one diamond in a Type 1 string of diamonds that replaces $h_mh_n$, and again $s_i$ and $y_i$ receive different colors from $\{1_a,1_b\}$. Without loss of generality, we may assume $s_i$ receives $1_a$ and $y_i$ receives $1_b$. Now, we assign to each vertex in $G_i - \{x_1^i, u_i, w_i\}$ the same color as it was assigned in $\widetilde{G_i}$, we assign $u_i$ the color $1_b$, $w_i$ the color $1_a$, and $x_1^i$ the color $2_a$.
\smallqed

\medskip

For an arbitrary non-leaf vertex $g_j\in V(T_G)$ with $g_i$ as its child, where $g_j\in A_k$, the constructions in Claim 3 provide an extension of the $(1,1,2,2)$-coloring of the subgraph of $G$ consisting of all $G_m$, where $g_m$ belongs to $A_0\cup\cdots\cup A_k$, to the subgraph $G_i$. (The constructions are dealing with four possible cases, notably $G_i$ is a triangle, a diamond or a claw-free, 2-edge-connected, subcubic graph with either an odd or an even number of vertices of degree 2 and all other vertices of degree $3$.) By the described $(1,1,2,2)$-coloring of the vertices of $G_0$, and the mentioned extension of the $(1,1,2,2)$-coloring, we derive that $G$ admits a $(1,1,2,2)$-coloring.
\end{proof}

Theorem~\ref{thm:main} provides a construction of a $(1,1,2,2)$-coloring of a claw-free cubic graph. While this is a large class of subcubic graphs, the problem whether all subcubic graphs are $(1,1,2,2)$-colorable remains an open challenge.

\section*{Acknowledgements}
B.B. acknowledges the financial support of the Slovenian Research and Innovation Agency (research core funding No.\ P1-0297 and projects N1-0285, J1-3002, and J1-4008). K.K. would like to thank the financial support of the FRC at Trinity College. Additionally, support for this research was provided by an AMS-Simons Research Enhancement Grant for Primarily Undergraduate Institution Faculty.


\begin{thebibliography}{99}

\bibitem{AK-2022}  S.~E.~Anderson and K.~Kuenzel,
Power domination in cubic graphs and Cartesian products,
Discrete Math.\ 345 (2022) 113113.


\bibitem{bkl} J.~Balogh, A.~Kostochka, X.~Liu,
Packing chromatic number of subdivisions of cubic graphs,
Graphs Combin.\ 35 (2019) 513--537.


\bibitem{BFKR} B.~Bre\v sar, J.~Ferme, S.~Klav\v zar and D.F.~Rall,
 A survey on packing colorings,
 Discuss.\ Math.\ Graph Theory 40 (2020) 923--970.

\bibitem{bgt}
B.~Bre\v sar, N.~Gastineau and O.~Togni,
Packing colorings of subcubic outerplanar graphs,
Aequationes Math.\ 94 (2020) 945--967.


\bibitem{bkrw-2017b}
  B.~Bre\v sar, S.~Klav\v zar, D.F.~Rall and K.~Wash,
 Packing chromatic number, $(1,1,2,2)$-colorings, and characterizing the Petersen graph,
  Aequationes Math.\ 91 (2017) 169--184.


\bibitem{gt-2016}
  N.~Gastineau and O.~Togni,
$S$-packing colorings of cubic graphs,
  Discrete Math.\  339 (2016) 2461--2470.

\bibitem{goddard-2008}
  W.~Goddard, S.M.~Hedetniemi, S.T.~Hedetniemi, J.M.~Harris and D.F.~Rall,
 Broadcast chromatic numbers of graphs,
  Ars Combin.\  86 (2008) 33--49.

\bibitem{goddard-2012}
  W.~Goddard and H.~Xu,
The $S$-packing chromatic number of a graph,
  Discuss.\ Math.\ Graph Theory 32 (2012) 795--806.

\bibitem{Kostochka-2021}
A.~Kostochka and  X.~Liu,
Packing (1,1,2,4)-coloring of subcubic outerplanar graphs,
 Discrete Appl.\ Math.  302  (2021) 8--15.
		

\bibitem{kramer}
F.~Kramer and H.~Kramer,
A survey on the distance-colouring of graphs,
Discrete Math.\  308 (2008) 422--426.

\bibitem{liu-2020}  R.~Liu, X.~Liu, M.~Rolek and G.~Yu,
Packing $(1, 1, 2, 2)$-coloring of some subcubic graphs,
Discrete Appl.\ Math.\ 283 (2020) 626--630.

\bibitem{mt-2024} M.~Mortada and O.~Togni,
About S-packing coloring of subcubic graphs,
Discrete Math.\ 347 (2024) 113917.

\bibitem{Oum} S. Oum,
Perfect matchings in claw-free cubic graphs,
Electron.\ J.\ Combin.\ 18 (2011) P.62.


\bibitem{Petersen} J.~Petersen,
Die Theorie der regul\"{a}ren Graphen,
Acta Math.\ 15 (1891) 193--220.


\bibitem{Plesnik} J. Plesnik,
Connectivity of regular graphs and the existence of $1$-factors,
Mat.\ Cas.\ Slov.\ Akad.\ Vied.\ 22 (1972) 310--318.


\bibitem{tt-2024} B.~Tarhini and O.~Togni,
S-packing coloring of cubic Halin graphs,
Discrete Appl.\ Math.\ 349 (2024) 53--58.

\bibitem{YW} W. Yang and B.~Wu,
On packing S-colorings of subcubic graphs,
Discrete Appl.\ Math.\ 334 (2023) 1--14.

\end{thebibliography}
\end{document}